\documentclass[a4paper]{amsart}
\usepackage{amssymb,amscd, dsfont, mathtext}
\usepackage[cp1251]{inputenc}
\usepackage[T2A]{fontenc}
\usepackage{graphicx}
\usepackage[all]{xy}

\usepackage{ mathrsfs }

\usepackage{latexsym}

\newtheorem{theorem}{Theorem}[section]
\newtheorem{Proposition}[theorem]{Proposition}

\newtheorem{Lemma}[theorem]{Lemma}

\theoremstyle{definition}

\newtheorem{example}[theorem]{Example}
\newtheorem{construction}[theorem]{Construction}

\theoremstyle{remark}

\newtheorem*{remark}{Remark}

\numberwithin{equation}{section}

\def\Z{\mathbb Z}
\newcommand{\rzk}{\mathcal R_{\mathcal K}}

\def\sK{\mathcal K}

\def\sK{\mathcal K}

\def\phi{\varphi}

\newcommand{\mb}[1]{{\textbf {\textit#1}}}
\newcommand{\bl}{\lambda\kern-0.53em\lambda}
\newcommand{\bmu}{\mu\kern-0.55em\mu}
\newcommand{\bnu}{\nu\kern-0.51em\nu}

\def\lim{\mathop\mathrm{lim}\nolimits}

\newcommand{\zk}{\mathcal Z_{\mathcal K}}

\begin{document}

\title[Free commuting involutions
	 \\ on closed two-dimensional surfaces]{Free commuting involutions on closed two-dimensional surfaces}
\author{Tatiana Neretina}
\email{nertata2@yandex.ru}
\address{Lomonosov Moscow State University}

\maketitle

\begin{abstract}
	We consider the function $f(g)$ that assigns to an orientable surface $M$ of genus $g$ the maximal number of free commuting independent involutions on $M$. We show that the surface of minimal genus $g$ with $f(g)=n$ is a real moment-angle complex $\rzk$, where $\sK$ is the boundary of an $(n+2)$-gon. The genus is given by the formula $g = 1 + 2^{n-1}(n-2)$.
\end{abstract}

\section{Introduction}

One of the main objects of study in toric topology is the moment-angle-complex~$ \mathcal Z_{\sK} $,  which is a cell complex with a torus action constructed  from a simplicial complex $\sK$, see~\cite{bu-pa15}. Along with the moment-angle complex $\zk$, its real analog $\rzk$ is considered and has many interesting properties. If $\sK$ is a simplicial subdivision of an $(n-1)$-dimensional sphere with $m$ vertices, then $\zk$ is an $(m+n)$-dimensional (closed) manifold, and $\rzk$ is an $n$-dimensional manifold. In particular, if $\sK$ is the boundary of an $m$-gon, then $ \rzk $ is an orientable closed two-dimensional manifold (a surface). Given a positive integer~$g$, let $f(g)$ be the maximal number of free commuting independent involutions on a closed orientable surface $M_g$ of genus $g$ (that is, $f(g)$ is the largest possible $ n $ such that $ (\Z_2)^n$ acts on $M_g$ freely). In this paper, it is proved that the function $f(g)$ attains a local maximum on the surfaces $\rzk$ see~Fig.~1.
In other words, for any integer $n$, the surface of minimum genus $g$ that supports a free action of  $\Z_2^n$ has the form $\rzk$, where $\sK$ is an $(n+2)$-gon. The genus is given by the formula $g = 1 + 2^{n-1}(n-2)$.

The author is grateful to her supervisor Taras Panov for suggesting the problem and attention to this work.

\section{Basic notions and preliminary statements}

A moment-angle complex is a special case of the following construction:

\begin{construction}[polyhedral product]\label{nsc}
Let $\sK$ be a simplicial complex on the set 
 $[m]={1, 2, \ldots, m}$ and let
$$
  (\mb X,\mb A)= \{(X_1,A_1),\ldots,(X_m,A_m)\}
$$
be a collection of $m$ pairs of topological spaces, $A_i\subset X_i$.
For each subset $I\subset[m]$ denote
\begin{equation*}
  (\mb X,\mb A)^I=\bigl\{(x_1,\ldots,x_m)\in
  \prod_{j=1}^m X_j\colon\; x_j\in A_j\quad\text{for }j\notin I\bigl\}
\end{equation*}
and define the \textbf{polyhedral product} of $(\mb X,\mb A)$ corresponding to a simplicial complex $\sK$ by
\[
  (\mb X,\mb A)^{\sK}=\bigcup_{I\in\mathcal K}(\mb X,\mb A)^I=
  \bigcup_{I\in\mathcal K}
  \Bigl(\prod_{i\in I}X_i\times\prod_{i\notin I}A_i\Bigl).\notag
\]
Here the union is a subset of $\prod_{j=1}^m X_j$.

In the case when all pairs $(X_i,A_i)$ are the same, i.e.
$X_i=X$ and $A_i=A$ for $i=1,\ldots,m$, we use the notation
$(X,A)^\sK$ for $(\mb X,\mb A)^\sK$.
\end{construction}
If $(X, A) = (D^2, S^1)$ then $\mathcal Z_{\sK} = (D^2, S^1)^\sK$ is called the \emph{moment-angle complex}. We will consider its real analogue:
\begin{equation*}
  \rzk = ([-1, 1], \{-1, 1\})^{\sK},
\end{equation*}
where $[-1, 1] = D^1$~ is the a line segment, and $\{-1, 1\} = \partial[-1, 1]$~ is the pair of points.

If $\sK$~ is a simplicial decomposition of $(n-1)$-dimensional sphere, then $\rzk$ is an $n$-dimensional manifold  with the action of group $(\Z_2)^m$ (see \cite[Theorem~4.1.7]{bu-pa15}).
In other words, we have $m$ commuting involutions on~$\rzk$. Of these $m$ involutions no more than  $m-n$ act freely.
For completeness, below we present a proof of these facts for $\sK$ being the boundary of an $m$-gon. The surfaces $\rzk$ corresponding to $m$-gons appear as ``regular topological skew polyhedra'' in the 1937 work of Coxeter ~\cite{coxe37}.
\begin{Proposition}\label{g}
Let $\sK$~ be the boundary of an $m$-gon. Then $\rzk$~ is an orientable closed surface of genus
$$g = 1+2^{m-3}(m-4). $$
\end{Proposition}
\begin{proof}
Consider the cube $I^m = [-1, 1]^m$ with the standard structure of cubical complex. Its two-dimensional skeleton consists of faces of the form
\begin{multline*}
 \{(\varepsilon_1, \dots, \varepsilon_{i-1}, x_i, \varepsilon_{i+1},\dots ,\varepsilon_{j-1}, x_j, \varepsilon_{j+1}, \dots, \varepsilon_m)\colon x_i, x_j \in [-1, 1]\},\\
 {\text{where} \; \varepsilon_1, \dots, \varepsilon_m =\pm 1}.
\end{multline*}
By definition, $\rzk$ is a cubical subcomplex of $I^m$. Indeed,  $\rzk$ is a union of two-dimensional
faces of cube of the following form:
\begin{equation*}
 \{(\varepsilon_1, \dots, \varepsilon_{i-1}, x_i, x_{i+1}, \varepsilon_{i+2}, \dots, \varepsilon_m)\colon x_i, x_{i+1} \in [-1, 1]\},\  \text{where} \; \varepsilon_1, \dots, \varepsilon_m =\pm 1
\end{equation*}
(here and below we consider subscripts modulo  $m$, i.e., $m+1 \equiv 1$).
Each one-dimensional face
 $\{(\varepsilon_1, \dots, x_i, \dots, \varepsilon_m)\colon x_i \in [-1, 1]\}$ is the boundary of precisely two squares, namely $(x_i, x_{i+1}) \in [-1, 1]^m$ and $(x_{i-1}, x_i) \in [-1, 1]^m$.
 Moreover, the intersection of two squares is either empty, or a vertex, or a one-dimensional face (since this is valid for the whole cube and one-dimensional faces of $\rzk$ form the whole one-dimensional skeleton of the cube).
 Each vertex is contained in precisely $m$ squares. This implies that $\rzk$ is a closed two-dimensional orientable manifold.

The surface $\rzk$ is glued of $F = m2^{m-2}$ squares. Each square has $4$ edges, and each edge is contained in two $2$ squares. Therefore the number of edges is $ E = 4  F / 2 = 2  F$. All vertices of the cube
are vertices of our surface, therefore the number of vertices is $V = 2^m$.
Thus, the Euler characteristic is
\begin{equation}\label{eul}
 \chi(\rzk) = V - E + F = V - 2F + F = V - F = 2^{m-2}(4-m),
\end{equation}
which immedially implies the formula for genus of the surface.
\end{proof}

\begin{example}\label{eeee}
Let  $\sK$ be the boundary of a triangle. Then $\rzk$ is the boundary of a $3$-dimensional cube and therefore is homeomorphic to a sphere, i.e., the genus is~$0$.
\end{example}

\begin{Lemma}\label{sv}
If $\sK$ is the boundary of an  $m$-gon, then there ia a free action of $\Z_2^{m-2}$ on $\rzk$.
\end{Lemma}
\begin{proof}
We describe a freely acting subgroup $\Z_2^{m-2} \subset \Z_2^m$ by explicitly defining its generators. Let  $\psi_i$ be the involution sending $x_i$ to $-x_i$ and fixing the other coordinates. The involutions  $\psi_i$, $i = 1, \ldots, m$, commute and therefore generate a $\Z_2^m$-action. However, this action is not free as $\psi_i$ fixes the points whose $i$-th coordinate is zero. Consider the composition $\psi_i \circ \psi_j$, where $|i-j| > 1$ (it corresponds to a diagonal in the polygon  $\sK$). The  set of fixed points of the involution $\psi_i \circ \psi_j$ is the set with coordinates  $x_i = x_j = 0$,  but $\rzk$ does not contain such points, since  $x_i, x_j$ are not consecutive. If $m = 2k$, then consider  the $m-2$ involutions $\psi_1 \circ \psi_3$, $\psi_3 \circ \psi_5$, $\dots$, $\psi_{2k-3} \circ \psi_{2k-1}$ and $\psi_2 \circ \psi_4$, $\psi_4 \circ \psi_6$, $\dots$, $\psi_{2k-2} \circ \psi_{2k}$.
These involutions  commute pairwise and generate a free action of $\Z_2^{m-2}$. Similarly, for an odd  $m = 2k+1$ consider  the $m-2$ involutions $\psi_1 \circ \psi_3$, $\psi_3 \circ \psi_5$, $\dots$, $\psi_{2k-3} \circ \psi_{2k-1}$, $\psi_2 \circ \psi_4$, $\psi_4 \circ \psi_6$, $\dots$, $\psi_{2k-2} \circ \psi_{2k}$ and $\psi_1 \circ \psi_{2k} \circ \psi_{2k+1}$.
\end{proof}

\begin{remark}
For even $m$, each element of the freely acting group $\Z_2^{m-2}$ defined above preserves the orientation of $\rzk$, since it is a composition
of an even number of elementary involutions $\psi_i$. For odd $m$, the involution $\psi_1 \circ \psi_{2k} \circ \psi_{2k+1}$ reverses the orientation of~$\rzk$.
\end{remark}

Next, consider an arbitrary closed two-dimensional surface $M$. We ask the following question: find the maximal  $n$ such that there is a free action of $\Z_2^n$ on~$M$. Obviously, $n$ depends only on the genus  $h$ of the surface $M$.
(The Euler characteristic of an orientable surface of genus  $h$ is $2-2h$, and the Euler characteristic of a nonorientable surface of genus  $h$ is $2-h$.) Define the function 
$$
f(h) = \max\{n \colon \text{there is an action of $\Z_2^n$  on a surface of genus } h\}.
$$
Let $f(h) = n$. Then $B = M/\Z_2^n$ is a closed two-dimensional manifold with Euler characteristic 
$\chi(B) = \chi(M)/ 2^n$.

\begin{Proposition}\label{prop}
Let $g$ be the genus of  $B= M/\Z_2^n$. If the surface  $B$ is orientable, then  $n \leqslant 2g$. If $B$ is nonorientable, then  $n \leqslant g$.
\end{Proposition}
\begin{proof}
Consider the case of orientable $B$. The fundamental group of $B$ is
$$
G = \pi_1(B)=  \langle a_1, b_1, \dots, a_g, b_g \; | \; a_1b_1a_1^{-1}b_1^{-1} \dots a_gb_ga_g^{-1}b_g^{-1} = 1\rangle.
$$
The fundamental group $\pi_1(M) = H$ of the manifold $M$ is a normal subgroup of the group $G$ (since the action of  $\Z_2^n$ is free) and
$$
G/H \thickapprox \Z_2^n.
$$
Therefore the square of each coset is 1,
$$
[a_i]^2 = [b_j]^2 = 1,
$$
for all $i, j = 1,..., g$. Moreover, all cosets (i.e., elements of $G/H \thickapprox \Z_2^n$) commute. The quotient group $G/H$ is generated by the cosets  $[a_1], [b_1], \dots, [a_g], [b_g]$. Since they commute and their squares equal the unit, we can compose at most  $2^{2g}$ words of them. Therefore the order of the group $G/H\thickapprox  \Z_2^n$ is at most $ 2^{2g}$. Hence, $n \leqslant 2g$.
For a nonorientable surface $B$ the argument is similar: the fundamental group
$$
G = \pi_1(B)=  \langle a_1, \dots, a_g | \; a_1^2 \dots a_g^{2} = 1\rangle
$$
has $g$ generators and the order of the quotient $G/H$ is at most $ 2^g$.
\end{proof}

\begin{Proposition}\label{prop2}\ \begin{itemize}
                       \item[a)] For any orientable surface $B$  of genus $g$ and an integer $n \leqslant 2g$, there exists an orientable surface $M$ with a free action of $\Z_2^n$ such that $B= M/\Z_2^n$.
                       \item[b)] For any nonorientable surface $B$ of genus $g$ and an integer $n \leqslant g$, there exists a surface  $M$ with a free action of $\Z_2^n$ such that $B= M/\Z_2^n$.
                     \end{itemize}
\end{Proposition}
\begin{proof}
Let $B$ be orientable. Consider the subgroup $P$ in the fundamental group  $\pi(B)=G$ generated by the squares of all elements
$$
   P = \langle g^2, g \in G\rangle,
$$
and consider its normalizer subgroup $H = GPG^{-1}$. Then $H$ is a normal subgroup and $H\neq G$, since $H$ contains only elements with even number of letters. The relations  $[a_i]^2 = [b_j]^2 = [a_i b_j a_i b_j] = 1$ in the quotient group $G/H$ imply that $G/H \cong \Z_2^{2g}$. Hence there exists a regular covering of $B$, which gives a free action of $G/H \cong \Z_2^{2g}$  (see \cite{hatc02}). Now add one generator to $P$: $P_1 = \langle P, a_1 \rangle $ and consider  $H_1 = GP_1G^{-1}$. Then $H_1$ is a proper normal subgroup of $G$ (since in each of its elements the generator $b_g$ occurs even number of times), and the quotient group is  $G/H_1 \cong \Z_2^{2g-1}$. The corresponding covering of $B$ is regular and gives a free action of $\Z_2^{2g-1}$. Continuing this process, we  add to $P$ the other generators $a_i, b_i$, and obtain regular coverings of $B$ corresponding to free actions of $\Z_2^{2g-k}$ for all $k = 1, 2,..., 2g-1$.

The nonorientable case is considered similarly.
\end{proof}

\section{Main results}
Let $M_g$ be a surface of genus $g$.
\begin{Proposition}\label{3}
Let $n$ be the maximal integer such that
\begin{equation}\label{nmax}
\chi(M_g) = a \cdot 2^n \qquad \text{with } \; n \leqslant 2 - a,\; a \leqslant 1.
\end{equation}
Then $f(g) = n$ if $a$ is even, and $n-1 \leqslant f(g) \leqslant n$ if $a$ is odd.
\end{Proposition}
\begin{proof}
Let $\Z^k_2$ act freely on a surface $M_g$. Then $\chi(M_g)=a' \cdot 2^k $ where $a' =\chi(M_g/~\Z^k_2)$, and $k \leqslant 2 - a'$ by Proposition \ref{prop}. By definition, $n$ is the maximum of such $k$, hence $f(g)\leqslant n$. Now we estimate $f(g)$ from below. Consider the two cases:

\smallskip

\noindent    \emph{Case 1: $a$ is even.}
  Consider an orientable surface $B$ with Euler characteristic  $a$ and genus $\frac{2-a}{2}$. Since $n \leqslant 2\frac{2-a}{2}$, by Proposition \ref{prop2} there exists an orientable surface  $M$ such that $\Z^n_2$ acts freely on $M$ and $M = B/\Z^n_2$. The Euler characteristic is  $\chi(M) = a \cdot 2^n = \chi(M_g)$, therefore  $M$ has genus  $g$ and $f(g)\geqslant n$.
  
  \smallskip
  
 \noindent    \emph{Case 2: $a$ is odd.} Write $\chi(M) = 2a \cdot 2^{n-1}$. Consider an orientable surface $B$ of Euler characteristic $2a$, i.e., of genus $\frac{2-2a}{2}$. We have  $n-1 \leqslant 2 - 2a$, since $n \leqslant 2- a$ and $a \leqslant 1$. By Proposition \ref{prop2} there exists an orientable surface  $M$ such that $\Z^{n-1}_2$ acts freely on $M$ and $M = B/\Z^{n-1}_2$. The Euler characteristic is   $\chi(M) = 2a \cdot 2^{n-1} = \chi(M_g)$, therefore $M$ has genus $g$ and $f(g)\geqslant n-1$.
\end{proof}

%

Define $H(g)$ as the inverse function to  $g(x) = 1 + 2^{x-1} (x-2) $, that is,
$$
    H(g) = \frac{W\left(\frac{1}{2} (g-1) \ln2\right)}{\ln 2}+2,
$$
where $W$ denotes the Lambert function (i.e., the function inverse to  $x e^x$).
\begin{theorem}{Theorem}
  $f(g) \leqslant H(g)$, and  an equality is attained on real moment-angle manifolds and only on them.
\end{theorem}
\begin{proof}
Substituting $\chi(M_g) = 2 - 2g$ and $a = 2-b$ in \eqref{nmax} we obtain
$$
g = 1 + 2^{n-1}(b - 2), \quad \text{where}\; n \leqslant b,\; b \geqslant1.
$$
By Proposition \ref{3}, for the maximal $n$ satisfying the conditions above we have $f(g) \leqslant n$. On the other hand we have
$$
g(n) = 1 + 2^{n-1}(n - 2)\leqslant 1 + 2^{n-1}(b - 2) = g,
$$
or equivalently $ n \leqslant H(g)$. This implies the required inequality $f(g) \leqslant H(g)$.
An equality is attained if and only if the genus $g$ can be written as
$
g = 1 + 2^{n-1}(n - 2).
$
Compairing this expession with the formula from Proposition \ref{g}, we obtain that  $M\cong \rzk$, where $\sK$ is the boundary of  $(n+2)$-gon.
\end{proof}
The values of the function $f(g)$ are shown  in Fig.~1, where the dashed line is the graph of $H(g)$.


\begin{figure}[h]
			\center{\includegraphics[width=128mm]{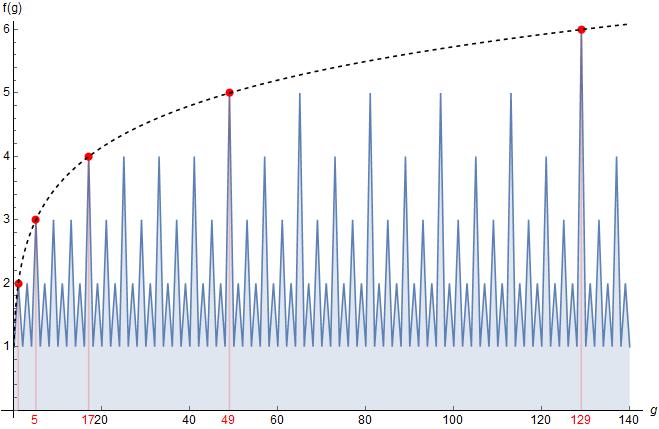}}
\caption{The graph of the function $f(g)$, where $g$ is  the genus of a surface.}
\end{figure}


\end{document}